\documentclass{amsart}

\usepackage{amsmath}
\usepackage{amssymb}
\usepackage{amscd}
\usepackage{lscape}
\usepackage[all]{xy}

\newtheorem{theorem}{Theorem}[section]
\newtheorem{lemma}[theorem]{Lemma}

\newtheorem{corollary}[theorem]{Corollary}
\newtheorem{proposition}[theorem]{Proposition}

\theoremstyle{definition}
\newtheorem{definition}[theorem]{Definition}
\newtheorem{example}[theorem]{Example}

\theoremstyle{remark}
\newtheorem{remark}[theorem]{Remark}

\numberwithin{equation}{section}

\newcommand{\GL}{\operatorname{GL}}
\newcommand{\Gr}{\operatorname{Gr}}
\newcommand{\Fl}{\operatorname{Fl}}
\newcommand{\Grep}{\operatorname{Grep}}
\newcommand{\Frep}{\operatorname{Frep}}

\newcommand{\module}{\operatorname{mod}}

\newcommand{\Rep}{\operatorname{Rep}}
\newcommand{\Mod}{\operatorname{Mod}}
\newcommand{\Hom}{\operatorname{Hom}}

\newcommand{\Aut}{\operatorname{Aut}}
\newcommand{\End}{\operatorname{End}}
\newcommand{\Ext}{\operatorname{Ext}}

\newcommand{\Exp}{\operatorname{Exp}}
\newcommand{\Log}{\operatorname{Log}}

\newcommand{\Id}{\operatorname{Id}}

\newcommand{\D}[1]{\mathcal{#1}}
\newcommand{\B}[1]{\mathbb{#1}}
\newcommand{\SI}{\operatorname{SI}}
\newcommand{\SIR}{\operatorname{SIR}}
\newcommand{\Spec}{\operatorname{Spec}}
\newcommand{\Proj}{\operatorname{Proj}}

\newcommand{\br}[1]{\overline{#1}}
\newcommand{\innerprod}[1]{\langle#1\rangle}
\renewcommand{\ss}[2]{^{{#1}\cdot{\rm #2}}}
\newcommand{\sm}[1]{\left(\begin{smallmatrix}#1\end{smallmatrix}\right)}

\def\Xint#1{\mathchoice
{\XXint\displaystyle\textstyle{#1}}%
{\XXint\textstyle\scriptstyle{#1}}%
{\XXint\scriptstyle\scriptscriptstyle{#1}}%
{\XXint\scriptscriptstyle\scriptscriptstyle{#1}}%
\!\int}
\def\XXint#1#2#3{{\setbox0=\hbox{$#1{#2#3}{\int}$}
\vcenter{\hbox{$#2#3$}}\kern-.5\wd0}}

\makeatletter
\def\equalsfill{$\m@th\mathord=\mkern-7mu
\cleaders\hbox{$\!\mathord=\!$}\hfill
\mkern-7mu\mathord=$}
\makeatother

\begin{document}

\title{Counting using Hall Algebras I. Quivers}

\author{Jiarui Fei}
\address{Department of Mathematics, University of California, Riverside, CA 92521, USA}
\email{jiarui@math.ucr.edu}
\thanks{}


\date{}
\dedicatory{}

\begin{abstract} We survey some results on counting the rational points of moduli spaces of quiver representations. We then make generalizations to Grassmannians and flags of quiver representations. These results have nice applications to the cluster algebra. Along the way, we use the full Hopf structure of the Hall algebra of a quiver.
\end{abstract}

\maketitle

\section*{Introduction}
In this series of notes, we try to develop an algorithm to count the points of representation varieties of basic algebras, and their moduli, Grassmannians, and so on. This seems to be a mission impossible, since any affine variety can be realized as a representation variety and any projective variety can be realized as a moduli of representations. I know the former fact from Birge Huisgen-Zimmermann and the latter from Lutz Hille.

To start this project, we first deal with the simplest case when the algebra is hereditary. Many beautiful results have been discovered by M. Reineke. The most important idea of his is to apply certain algebra morphism to various identities in the Hall algebra. The algebra morphism bears different names in the literature, to name a few, integration map and evaluation map. In our vocabulary, we call this type of morphisms Hall characters. In this note, we first systematically review several known results. We try to concentrate several papers into few pages, hoping that readers can penetrate into the Tao and quickly master the technique.

Hall algebra of a finitary hereditary category is in fact a bialgebra, so we can compose the known characters with the comultiplication to get new characters (Example \ref{Ex:double}). The most important identity above all in the Hall algebra may be the one related to the Harder-Narasimhan filtration (Lemma \ref {L:HNid}). Apply this new character to it, we get a counting formula for Grassmannians of quiver representations (Corollary \ref{C:QG}), which we think is mildly new. This formula has an interesting application in the cluster algebra. We derive a quantum cluster multiplication formula, which is obtained independently in \cite{DX}. We greatly simplify the proof of the positivity theorem in acyclic cases \cite[Theorem 1]{CR}. Many things above can be generalized to flags of quiver representations and beyond.


In this notes, we first collect some useful facts from \cite[Appendix]{CV} on points counting. Then in section \ref{S:basic}, we review some basics on the quiver representation, and related varieties. In section \ref{S:def}, we review the Ringel-Hall algebra and give our definition and examples on the Hall character. In section \ref{S:moduli}, we follow Reineke to count the points of moduli of quiver representations in a simple case. We then proceed to the case of Grassmannians and flags of quiver representations (Section \ref{S:grass},\ref{S:flag}). Corollary \ref{C:QG}, Theorem \ref{T:QG}, and Theorem \ref{T:QF} are just analogues of Corollary \ref{C:int1} and Corollary \ref{C:coprime1}. As an application, in section \ref{S:cluster} we give a very simple proof of the positivity theorem in acyclic cases (Corollary \ref{C:pos}) and derive a quantum cluster multiplication formula (Proposition \ref{P:CM}). In section \ref{S:functor}, we make generalization (Theorem \ref{T:morphism}) in a functorial direction. In Section \ref{S:Hopf}, we use the antipode to find the inverse of group-like elements. Lemma \ref{L:chiinv} is crucial to the last section. For completeness, the last section is a short survey for related generating series. We review an interesting result of \cite{MR} relating two different kinds of series. Finally, if one rewrites certain generating series in a proper way, then one gets the so-called quantum dilogarithm identities.

The first six sections were already written in 2009. We did not want to publish it until we feel that some subsequential work can be done. Although our original motivation came from the cluster algebra, now we are more inclined to settle our simple mind on counting. So we thank Fan Qin \cite{Q} for saving us a lot of space in Section \ref{S:cluster}. Sophisticated minds are always welcome to play in the motivic ground.

Let $X$ be a variety over $k=\B{F}_q$ and $X_{\br{k}}=X\otimes_k\br{k}$. We denote by $H_c^i(X,\B{Q}_l)$ the $i$-th $l$-adic cohomology group with compact support of $X_{\br{k}}$. The key method for counting rational points on $X$ is given by the Grothendieck-Lefschetz trace formula:
$$|X(\B{F}_{q^r})|=\sum_{i=0}^{2\dim X}(-1)^i Tr(F^r;H_c^i(X,\B{Q}_l)),$$
where $F$ is the Frobenius morphism $X_{\br{k}}\to X_{\br{k}}$.
$X$ is called {\em $l$-pure} if the eigenvalues of $F$ on $H_c^i(X,\B{Q}_l)$ have absolute value $q^{i/2}$. It is known that if $X$ is smooth and proper over $\br{k}$ then $X$ is $l$-pure.

$X$ is called {\em polynomial-count} (or has a counting polynomial) if there exists a (necessarily unique) polynomial $P_X=\sum a_i t_i\in\B{C}[t]$ such that for every finite extension $\B{F}_{q^r}/\B{F}_q$, we have $|X(\B{F}_{q^r})|=P_X(q^r)$.

\begin{lemma} {\em \cite[Proposition 6.1]{R2}} \label{L:Euler}
If $X$ is polynomial-count, its counting polynomial $P_X$ must lie in $\B{Z}[t]$. Its specialization at $q=1$ gives the $l$-adic Euler characteristic of $X_{\br{k}}$.
\end{lemma}

\begin{definition} The {\em Poincare polynomial} $P(X,q)\in\B{Z}[q^{1/2}]$ of $X$ is
$$P(X,q)=\sum_{i\geq 0} (-1)^i\dim H_c^i(X,\B{Q}_l)q^{i/2}.$$
\end{definition}

We will frequently use the following lemma from \cite[Appendix]{CV}
\begin{lemma} \label{L:polycount} Assume that $X$ is $l$-pure and polynomial-count. Then $P_X(q)=P(X,q)$.
In particular, $P_X(t)\in\B{N}[t]$.
\end{lemma}

\section{Basics on Quiver Representation, Moduli, and Grassmannian}\label{S:basic}

We first provide some background on the moduli spaces of quiver representations. For a good introduction or more detailed treatment, we recommend \cite{Ki,R3}. Most constructions and results in this section are originally obtained assuming the base ring $k$ is an algebraically closed field. However, all the constructions can be naturally done over $\B{Z}$ and all proofs can be carried over almost word by word from the algebraically closed case \cite[Section 4,5]{M}. All the constructions over any field $k$ can be obtained by a change of base from $\B{Z}$. So let us assume our base ring $k$ to be $\B{Z}$.

Let $Q$ be a finite quiver with the set of vertices $Q_0$ and the set of arrows $Q_1$. If $a\in Q_1$ is an arrow, then $ta$ and $ha$ denote its tail and its head respectively. Fix a {\em dimension vector} $\alpha$, the space of all $\alpha$-dimensional representations is
$$\Rep_\alpha(Q):=\bigoplus_{a\in Q_1}\Hom(k^{\alpha(ta)},k^{\alpha(ha)}).$$
The group $G=\GL_\alpha:=\prod_{v\in Q_0}\GL_{\alpha(v)}$ acts on $\Rep_\alpha(Q)$ by the natural base change. Two representations $M,N\in\Rep_\alpha(Q)$ are isomorphic if they lie in the same $\GL_\alpha$-orbit. We denote the dimension vector of a representation $M$ by $\br{M}$.

A {\em weight} $\sigma$ is an integral linear functional on $\mathbb{Z}^{Q_0}$. Consider the {\em character} of $\GL_\alpha$:
$$\{g(v)\mid v\in Q_0\}\to\prod_{v\in Q_0}(\det g(v))^{\sigma(v)}.$$
We define the subgroup $\GL_\alpha^\sigma$ to be the kernel of the character map. The semi-invariant ring $\SIR_\alpha^\sigma(Q):=k[\Rep_\alpha(Q)]^{\GL_\alpha^\sigma}$ of weight $\sigma$
is $\sigma$-graded: $\oplus_{n\geqslant 0} \SI_\alpha^{n\sigma}(Q)$, where
$$\SI_\alpha^{\sigma}(Q):=\{f\in k[\Rep_\alpha(Q)]\mid g(f)=\sigma(g)f, \forall g\in\GL_\alpha\}.$$

A representation $M\in\Rep_\alpha(Q)$ is called {\em $\sigma$-semi-stable} if there is some non-constant $f\in \SIR_\alpha^\sigma(Q)$ such that $f(M)\neq 0$. It is called {\em stable} if the orbit $\GL_\alpha^\sigma M$ is closed of dimension $\dim\GL_\alpha^\sigma-1$. We denote the set of all $\sigma$-semi-stable (resp. $\sigma$-stable, $\sigma$-unstable) representations in $\Rep_\alpha(Q)$ by $\Rep_\alpha\ss{\sigma}{ss}(Q)$ (resp. $\Rep_\alpha\ss{\sigma}{st}(Q)$, $\Rep_\alpha\ss{\sigma}{un}(Q)$).
Based on Hilbert-Mumford criterion, King provides a simple criterion for the stability of a representation.
\begin{lemma} {\em \cite[Proposition 3.1]{Ki}}  \label{L:King} A representation $M$ is $\sigma$-semi-stable (resp. $\sigma$-stable) if and only if $\sigma(\br{M})=0$ and $\sigma(\br{L})\leqslant 0$ (resp. $<0$) for any non-trivial subrepresentation $L$ of $M$.
\end{lemma}

The GIT quotient with respect to a linearization of character $\sigma$ is
$$\Mod_\alpha^\sigma(Q):=\Proj(\oplus_{n\geqslant 0} \SI_\alpha^{n\sigma}(Q)),$$
which is projective over the ordinary quotient $\Spec(k[\Rep_\alpha(Q)]^{\GL_\alpha})$.
If we assume that $Q$ has no oriented cycles, then $k[\Rep_\alpha(Q)]^{\GL_\alpha}=k$, so the GIT quotient is projective.
The induced quotient map $q:\Rep_\alpha\ss{\sigma}{ss}(Q)\to\Mod_\alpha^\sigma(Q)$ is a {\em good categorical quotient} and the restriction of $q$ on $\Rep_\alpha\ss{\sigma}{st}(Q)$ is a {\em geometric quotient}.

There is a more general notion of stability. It is called $\mu$-stability for a {\em slope function} $\mu$, which by definition is $\sigma/\theta$ with $\sigma,\theta$ two weights and $\theta(\alpha)>0$ for any dimension vector $\alpha$.

\begin{definition} A representation $M$ is called $\mu$-semi-stable (resp. $\mu$-stable) if
$\mu(\br{L})\leqslant \mu(\br{M})$ (resp. $\mu(\br{L})<\mu(\br{M})$) for every non-trivial subrepresentation $L\subset M$.
\end{definition}

A slope function $\mu$ is called {\em coprime} to $\alpha$ if $\mu(\gamma)\neq \mu(\alpha)$ for any $\gamma<\alpha$. So if $\mu$ is coprime to $\alpha$, then there is no strictly semistable (semistable but not stable) representation of dimension $\alpha$. Note that the slope stability is related to the classical stability as follows. A representation $M$ is $\sigma$-(semi-)stable, then $M$ is also $(\sigma/\theta)$-(semi-)stable. Conversely, if $M$ is $(\sigma/\theta)$-(semi-)stable, and $\frac{\sigma(\br{M})}{\theta(\br{M})}=\frac{a}{b}$ where $a, b\in\B{Z}$ and $b>0$, then $M$ is $(b\sigma-a\theta)$-(semi-)stable. Note that the semi-stable objects with a fixed slope $\mu_0$ form an exact subcategory $\module_{\mu_0}(Q)$.

\begin{lemma} {\em \cite[Proposition 3.3]{DW2}}
\begin{enumerate}
\item[(1)]{\em Harder-Narasimhan filtration:}
Every representation $M$ has a unique filtration $0=M_0\subset M_1\subset\cdots\subset M_{m-1}\subset M_{m}=M$ such that $N_i=M_i/M_{i+1}$ is $\mu$-semi-stable and $\mu(\br{N}_i)>\mu(\br{N}_{i+1})$.\\
\item[(2)]{\em Jordan-Holder filtration:}
Every $\mu$-semi-stable representation $M$ has a filtration $0=M_0\subset M_1\subset\cdots\subset M_{m-1}\subset M_{m}=M$ such that $N_i=M_i/M_{i+1}$ is $\mu$-stable with the same slope. The set $\{N_i\}$ is uniquely determined.\\
\end{enumerate}
\end{lemma}

For dimension vectors  $\beta,\gamma$, let $\alpha=\beta+\gamma$. We define $\Gr\binom{\alpha}{\gamma}:=\prod_{v\in Q_0}\Gr\binom{\alpha(v)}{\gamma(v)}$, where $\Gr\binom{\alpha(v)}{\gamma(v)}$ is the usual Grassmannian variety of $\gamma(v)$-dimensional subspaces of $k^{\alpha(v)}$. We define the variety
$$\Grep\binom{\alpha}{\gamma}=\{(M,L)\in \Rep_{\alpha}(Q) \times\Gr\binom{\alpha}{\gamma}\mid L\text{ is a subrepresentation of }M\}.$$

\begin{lemma} \label{L:Grep} {\em \cite[Section 3]{So}}
$p:\Grep\binom{\alpha}{\gamma}\to \Gr\binom{\alpha}{\gamma}$ is a vector bundle with fibre
$$\bigoplus_{a\in Q_1}(\Hom(k^{\gamma(ta)},k^{\gamma(ha)})\oplus \Hom(k^{\beta(ta)},k^{\alpha(ha)})).$$
In particular, $\Grep\binom{\alpha}{\gamma}$ is smooth and irreducible with \begin{equation}
\dim\Grep\binom{\alpha}{\gamma}=\dim\Rep_\alpha(Q)+\innerprod{\gamma,\beta}.
\end{equation}
\end{lemma}

Now consider another projection $r:\Grep\binom{\alpha}{\gamma}\to\Rep_\alpha(Q)$.
\begin{definition} Let $r:\Grep\binom{\alpha}{\gamma}\to\Rep_\alpha(Q)$ be the other projection,
the {\em Grassmannian of quiver representation} $\Gr_\alpha(M)$ is the fibre $r^{-1}(M)$, and its subvarieties
$\Gr_L^N(M)$ is $\{L'\in\Gr_\alpha(M)\mid L'\cong L,M/L'\cong N\}$.
\end{definition}

A representation $T$ is called {\em rigid} if $\Ext_Q(T,T)=0$. It is shown in \cite{C} that the rigidity does not depend on the base field.
So if $k=\br{\B{F}}_q$, $\Gr_\gamma(T)$ is a general fibre of $r$, which is smooth by the Bertini theorem \cite[Theorem I.6.2.2]{Sh}.

There are projective varieties related to both the moduli and Grassmannian of quiver representations.
To construct them, we introduce the tensor product algebra $A_2(Q):=kQ\otimes kA_2$, where $A_2$ is the quiver of Dynkin type $A_2$.
The quiver $\hat{Q}$ of $A_2(Q)$ can be described as follows. We put quiver $Q$ horizontally and its copy $Q^c$ over $Q$, then draw for each vertex $v$ of $Q$ an arrow from its copy $v_c$ to $v$. We call the quiver $\hat{Q}$ the commutating quiver of $Q$. The algebra $A_2(Q)$ is the algebra of the quiver $\hat{Q}$ with the obvious commuting relations. A representation of $A_2(Q)$ consists of a triple $(L,M,f)$, where $L,M$ are representations of $Q$ and $f$ is a morphism from $L$ to $M$. A dimension vector of $(L,M,f)$ is of the form $(\gamma,\alpha)$, where $\gamma$ and $\alpha$ are the dimension vectors of $L$ and $M$. Let
$$\Rep_{\gamma\hookrightarrow\alpha}(A_2(Q)):=\{(L,M,f)\in \Rep_{(\gamma,\alpha)}(A_2(Q)) \mid f \text{ is injective}\},$$
then it is a principal $\GL_\gamma$-bundle over $\Grep\binom{\alpha}{\gamma}$. In particular, it is smooth and irreducible.

Assume that $\sigma(\alpha)=0$, and let $\mu=\sigma/\theta$, where $\theta=(1,1,\dots,1)$. Note that $M\in\Rep_\alpha(Q)$ is $\sigma$-(semi)stable if and only if $M$ is $\mu$-(semi)stable. We define a new slope function $\hat{\mu}=\hat{\sigma}/\hat{\theta}$ on $\hat{Q}$ as follows. The value of $\hat{\theta}$ on $Q_0$ and $Q^c_0$ are the same as the value of $\theta$ on $Q_0$. The value of $\hat{\sigma}$ on $Q_0$ is the same as $\sigma$, and its value on $Q^c_0$ is $\epsilon\theta$ for a very small positive number $\epsilon$.

We claim that if $M$ is $\sigma$-stable and $f$ is injective, then $(L,M,f)$ is $\hat{\mu}$-stable;, and if $M$ is $\sigma$-unstable or $f$ is not injective, then $(L,M,f)$ is $\hat{\mu}$-unstable. Suppose that $M$ is $\sigma$-stable and $f$ is injective. Let $(L',M',f|_{L'})$ be a subrepresentation of $(L,M,f)$, and $(\gamma',\alpha')$ be its dimension vector.
If $\alpha'<\alpha$, we have that $\hat{\mu}(L',M',f|_{L'})=\frac{\sigma(\alpha')+\epsilon\theta(\gamma')}{\theta(\alpha'+\gamma')}<0$ since $\epsilon$ is assume to be very small. In the meanwhile, $\hat{\mu}(L,M,f)=\frac{\sigma(\alpha)+\epsilon\theta(\gamma)}{\theta(\alpha+\gamma)}>0$. If $\alpha'=\alpha$ and thus $\gamma'<\gamma$, then $\hat{\mu}(\gamma',\alpha)=\frac{0+\epsilon\theta(\gamma')}{\theta(\alpha)+\theta(\gamma')}$, which is easily checked to be less than $\hat{\mu}(\gamma,\alpha)=\frac{0+\epsilon\theta(\gamma)}{\theta(\alpha)+\theta(\gamma)}$.
Conversely, suppose that $M$ is $\sigma$-unstable, and let $V$ be a subrepresentation destabilizing $M$, then $(0,V,0)$ also destabilizes $(L,M,f)$. If $f$ is not injective and let $K$ be its kernel, then one can easily check that $(K,0,0)$ destabilizes $(L,M,f)$. We have proved our claim. In particular, we have

\begin{lemma} \label{L:grepmoduli} If $\Rep_\alpha\ss{\sigma}{ss}(Q)$ contains exclusively $\sigma$-stable points, then so is the GIT-quotient $\Mod_{(\gamma,\alpha)}^{\hat{\mu}}(A_2(Q))$. Thus it is smooth and irreducible, parameterizing pairs $(M,L)$, where $M$ is $\sigma$-stable and $L$ its subrepresentation. The natural projection $\Rep_{(\gamma,\alpha)}(A_2(Q))$ $\twoheadrightarrow\Rep_\alpha(Q)$ induces a surjective map $\Mod_{(\gamma,\alpha)}^{\hat{\mu}}(A_2(Q))\to\Mod_\alpha^{\sigma}(Q)$, whose fibre over $M$ is exactly $\Gr_\gamma(M)$.
\end{lemma}



The above construction generalizes to the flag varieties of quiver representations. For any decomposition of dimension vector $\alpha=\sum_{i=1}^t\alpha_i$, we define $\Fl_{\alpha_t\cdots\alpha_1}:=\prod_{v\in Q_0}\Fl_{\alpha_t(v)\cdots\alpha_1(v)}$, where $\Fl_{\alpha_t(v)\cdots\alpha_1(v)}$ is the usual flag variety parameterizing flags of subspaces of dimension $\alpha_1(v)<\alpha_1(v)+\alpha_2(v)<\cdots<\alpha_1(v)+\cdots+\alpha_{t-1}(v)$ in $k^{\alpha(v)}$. To simplify the notation, we denote $\dot{\alpha}_i:=\sum_{j=1}^i\alpha_j$.

We define the variety:
$$\Frep_{\alpha_t\cdots\alpha_1}=\{(M,L_1,\dots,L_{t-1})\in \Rep_{\alpha}(Q) \times\Fl_{\alpha_t\cdots\alpha_1}\mid L_1\subset\cdots\subset L_t=M\}.$$
The following analogs of Lemma \ref{L:Grep} and Lemma \ref{L:grepmoduli} are easy exercises left for readers.

\begin{lemma}
$p:\Frep_{\alpha_t\cdots\alpha_1}\to \Fl_{\alpha_t\cdots\alpha_1}$ is a vector bundle with fibre
$$\bigoplus_{a\in Q_1}\bigoplus_{i=1}^{t}\Hom(k^{\alpha_{i}(ta)},k^{\dot{\alpha}_i(ha)}).$$
In particular, $\Frep_{\alpha_t\cdots\alpha_1}$ is smooth and irreducible.
\end{lemma}

\begin{definition} Let $r:\Frep_{\alpha_t\cdots\alpha_1}\to\Rep_\alpha(Q)$ be the other projection, the {\em flag variety of $M$} $\Fl_{\alpha_t\cdots\alpha_1}(M)$ is the fibre $r^{-1}(M)$, and its subvarieties
$\Fl_{N_t,\dots,N_1}(M)$ is $\{(L_1,\dots,L_{t-1})\in\Fl_{\alpha_t\cdots\alpha_1}(M)\mid L_{i}/L_{i-1}\cong N_i\}$, where $L_0=0$ and $L_t=M$.
\end{definition}

Let $A_t$ be the Dynkin quiver $\bullet\to\bullet\to\cdots\to\bullet$ of type $A_t$, and $A_t(Q)$ be the tensor algebra $kA_t\otimes kQ$.
We define a similar slope function $\hat{\mu}=\hat{\sigma}/\hat{\theta}$ on $\hat{Q}$ as before. The value of $\hat{\theta}$ on $Q_0$ and all its copies are the same as the value of $\theta$ on $Q_0$. The value of $\hat{\sigma}$ on $Q_0$ is the same as $\sigma$, and its value on any other copy of $Q_0$ is $\epsilon\theta$ for a very small positive number $\epsilon$.

\begin{lemma} If $\Rep_\alpha\ss{\sigma}{ss}(Q)$ contains exclusively $\sigma$-stable points, then the moduli $\Mod_{(\dot{\alpha}_1,\dots,\dot{\alpha}_{t})}^{\hat{\mu}}(A_t(Q))$ is smooth and irreducible, and its induced map to $\Mod_\alpha^{\sigma}(Q)$ from the natural projection is surjective with fibre $\Fl_{\alpha_t,\dots,\alpha_1}(M)$ over $M$.

\end{lemma}

\section{Hall Characters} \label{S:def}

From now on, we will assume our field $k$ to be the finite field $\B{F}_q$ and all modules are finite-dimensional. 
We write $\innerprod{-,-}_a$ for the usual additive Euler form, and $\innerprod{-,-}:=q^{\innerprod{-,-}_a}$ for the multiplicative one.
For any three $kQ$-modules $U,V$ and $W$ with dimension vector $\beta,\gamma$ and $\alpha=\beta+\gamma$, we define
the {\em Hall number}
$$F_{UV}^W:=|\Gr_V^U(W)|,$$
and for any module $M$, we denote $a_M:=|\Aut_Q(M)|.$
Let $H(Q)$ be the space of all formal (infinite) linear combinations of isomorphism classes $[M]$ in $kQ$-$\module$.
\begin{lemma} {\em \cite[Lemma 2.2]{Hu}}
The completed {\em Hall algebra} $H(Q)$ is the associative algebra with multiplication
$$[U][V]:=\sum_{[W]}F_{UV}^W[W],$$ and unit $\eta:\B{C}\mapsto \B{C}[0]$.
\end{lemma}

Note that the algebra $H(Q)$ is naturally graded by the dimension vectors:
$H(Q)=\oplus_{\alpha}H_\alpha(Q)$, where $H_\alpha(Q)$ is the subspace of all formal linear combinations of $[M]\in\module_\alpha(Q).$

\begin{definition}
For any algebra $R$, an $R$-character of $H(Q)$ is an algebra morphism $c:H(Q)\to R$.
\end{definition}

\begin{example} {\em \cite[Lemma 6.1]{R1}} \label{Ex:int} Let $\D{P}_{\innerprod{Q}}$ be the completed quantum polynomial algebra with multiplication given by $t^\alpha t^\beta=\innerprod{\alpha,\beta}^{-1}t^{\alpha+\beta}$.
Then $\int:[M]\mapsto \frac{1}{a_M}t^\alpha$ is an $\D{P}_{\innerprod{Q}}$-character of $H(Q)$. The fact follows easily from a formula of Riedtmann.

\begin{lemma} {\em \cite[Proposition 2.3]{Hu}} $F_{UV}^W=\frac{|\Ext_Q(U,V)_W|}{|\Hom_Q(U,V)|}\frac{a_W}{a_U a_V}$, where $\Ext_Q(U,V)_W$ is the subspace of $\Ext_Q(U,V)$ representing extensions with middle term $W$.
\end{lemma}
\end{example}

To see more examples, we need to explore the structure of $H(Q)$ further. Define a comultiplication $\Delta:H(Q)\to H(Q)\otimes H(Q)$ by $$\Delta([W])=\sum_{[U][V]} F_{UV}^W\frac{a_Ua_V}{a_W}[U]\otimes[V]=\sum_{[U][V]}\frac{|\Ext_Q(U,V)_W|}{|\Hom_Q(U,V)|}[U]\otimes[V].$$

\begin{lemma} {\em \cite[Lemma 2.4]{Hu}} $H(Q)$ is a coassociative coalgebra with comultiplication $\Delta$ and counit $\epsilon([M])=\delta_{M,0}$.
\end{lemma}

In general, the multiplication and comultiplication are not naturally compatible in the sense that if we treat $H(Q)$ as a vector space with symmetric monoidal structure, then $\Delta: H(Q)\to H(Q)\otimes H(Q)$ is not an algebra morphism. However, if we twist the multiplication on $H(Q)\otimes H(Q)$ by
$$(U_1\otimes V_1)\cdot(U_2\otimes V_2)=\innerprod{\br{U}_1,\br{V}_2}^{-1}U_1U_2\otimes V_1V_2.$$ Then it follows from Green's formula \cite[Proposition 2.7]{Hu} that

\begin{lemma} {\em \cite[Theorem 2.6]{Hu}} $H(Q)$ is a bialgebra.
\end{lemma}

%
%
%
%

\begin{example} \label{Ex:double}
Let $\D{P}_{\innerprod{Q^2}}=\D{P}_{\innerprod{Q}}\otimes \D{P}_{\innerprod{Q}}$ be the completed double quantum polynomial algebra with multiplication given by $(x^{\beta_1} y^{\gamma_1})(x^{\beta_2}y^{\gamma_2})=\frac{\innerprod{\gamma_1,\beta_2}}{\innerprod{\alpha_1,\alpha_2}}x^{\alpha_1}y^{\alpha_2}$, where $\alpha_1=\beta_1+\beta_2$ and $\alpha_2=\gamma_1+\gamma_2$. A direct calculation using Example \ref{Ex:int} can show that $\int\otimes\int:H(Q)\otimes H(Q)\to \D{P}_{\innerprod{Q^2}}$ is an algebra morphism.

\begin{proposition} \label{P:oint} $\int_\Delta=(\int\otimes\int)\circ\Delta$ is a $\D{P}_{\innerprod{Q^2}}$-character of $H(Q)$. We have a concrete formula: $\int_\Delta([W])={a_W}^{-1}\sum_\gamma|\Gr_\gamma(W)|x^{\alpha-\gamma}y^{\gamma}$ and the following identity: \begin{equation} \sum_{\gamma_1+\gamma_2=\gamma}\innerprod{\gamma_1,\beta_2}|\Gr_{\gamma_1}(U)||\Gr_{\gamma_2}(V)|=\sum_{[W]}\frac{|\Ext_Q(U,V)_W|}{|\Ext_Q(U,V)|}|\Gr_{\gamma}(W)|.
\end{equation}
\end{proposition}
\begin{proof}\begin{align*}\int_\Delta([W]) & =(\int\otimes\int)\circ\Delta([W])\\
& =\sum_{[U][V]} F_{UV}^W\frac{a_Ua_V}{a_W}\int[U]\otimes\int[V]\\
& =\sum_{[U][V]} F_{UV}^W{a_W}^{-1}x^{\alpha-\br{V}}y^{\br{V}}\\
& ={a_W}^{-1}\sum_\gamma|\Gr_\gamma(W)|x^{\alpha-\gamma}y^{\gamma}.
\end{align*}
\begin{align*}& \int_\Delta[U]\int_\Delta[V]=\int_\Delta([U][V])\\
\Leftrightarrow & \sum_{\gamma_1}|\Gr_{\gamma_1}(U)|x^{\alpha_1-\gamma_1}y^{\gamma_1}\sum_{\gamma_2}|\Gr_{\gamma_2}(V)|x^{\alpha_2-\gamma_2}y^{\gamma_2}=a_Ua_V \int_\Delta\sum_{[W]}F_{UV}^W[W]\\
\Leftrightarrow & \sum_\gamma\sum_{\gamma_1+\gamma_2=\gamma}\frac{\innerprod{\gamma_1,\beta_2}}{\innerprod{\alpha_1,\alpha_2}}|\Gr_{\gamma_1}(U)||\Gr_{\gamma_2}(V)|x^{\alpha-\gamma}y^\gamma=\sum_\gamma\sum_{[W]}\frac{|\Ext_Q(U,V)_W|}{|\Hom_Q(U,V)|}|\Gr_{\gamma}(W)|x^{\alpha-\gamma}y^\gamma\\
\Leftrightarrow &
\sum_{\gamma_1+\gamma_2=\gamma}\innerprod{\gamma_1,\beta_2}|\Gr_{\gamma_1}(U)||\Gr_{\gamma_2}(V)|=\sum_{[W]}\frac{|\Ext_Q(U,V)_W|}{|\Ext_Q(U,V)|}|\Gr_{\gamma}(W)|.
\end{align*}\end{proof}
\end{example}


\section{Application to Moduli of Quiver Representations} \label{S:moduli}

We fix a slope function $\mu$. For a dimension vector $\alpha$, let $\chi_\alpha=\sum_{\br{M}=\alpha}[M]$ and $\chi_\alpha^{\rm ss}=\sum_{M\in\module_\alpha^{\rm ss}(Q)}[M]$. Our convention is that they contain zero representation $[0]$.
Since $\frac{1}{a_M}=\frac{|\D{O}_M|}{|\GL_\alpha|}$, we have that $\int \chi_\alpha=\frac{|\Rep_{\alpha}(Q)|}{|\GL_{\alpha}|}x^\alpha$. We denote the rational function $\frac{|\Rep_{\alpha}(Q)|}{|\GL_{\alpha}|}$ by $r_\alpha(q)$.

The existence of the Harder-Narasimhan filtration yields the following identity in the Hall algebra $H(Q)$.
\begin{lemma} {\em \cite[Proposition 4.8]{R1}} \label{L:HNid} $$\chi_\alpha=\sum_{} \chi_{\alpha_1}^{\rm ss}\cdot\dots\cdot\chi_{\alpha_s}^{\rm ss},$$ where the sum running over all decomposition $\alpha_1+\cdots+\alpha_s=\alpha$ of $\alpha$ into non-zero dimension vectors such that $\mu(\alpha_1)<\cdots<\mu(\alpha_s)$. In particular, solving recursively for $\chi_\alpha^{\rm ss}$, we get
\begin{equation}\label{eq:HallID} \chi_\alpha^{\rm ss}=\sum_* (-1)^{s-1}\chi_{\alpha_1}\cdot\cdots\chi_{\alpha_s},
\end{equation}
where the sum runs over all decomposition $\alpha_1+\cdots+\alpha_s=\alpha$ of $\alpha$ into non-zero dimension vectors such that $\mu(\sum_{l=1}^k\alpha_l)<\mu(\alpha)$ for $k<s$.
\end{lemma}

Apply the Hall character $\int$ to the identity \eqref{eq:HallID}, then we obtain the formula:
\begin{corollary} {\em \cite[Corollary 6.2]{R1}} \label{C:int1}
$$\frac{|\Rep_\alpha^{\rm ss}(Q)|}{|\GL_\alpha|}=\sum_* (-1)^{s-1}\prod_{1\leqslant i<j\leqslant s}\innerprod{\alpha_i,\alpha_j}^{-1}\prod_{k=1}^s r_{\alpha_k}(q),$$
where the summation $*$ is the same as the above lemma.
\end{corollary}

We denote the above rational function by $r_\alpha^{\rm ss}(q)$. It follows from Lemma \ref{L:Euler} and \ref{L:polycount} that
\begin{corollary} {\em \cite[Theorem 6.7]{R1}} \label{C:coprime1} Assume that $\alpha$ is coprime to $\mu$, then $(q-1)r_\alpha^{\rm ss}(q)$ is actually an integral polynomial, which counts the rational points of $\Mod_\alpha^{\mu}(Q)$. If $Q$ has no oriented cycles, then $\Mod_\alpha^{\mu}(Q)$ is smooth and projective, so this polynomial has non-negative coefficients.
\end{corollary}

Without the coprime assumption, it is proven in \cite{R2} that $\Mod_\alpha^{\mu}(Q)$ still has polynomial counting property. This property also holds for absolutely stable representations of fixed dimension vectors \cite{MR}. We will continue this discussion in the last section.

\section{Application to Grassmannians of Quiver Representations} \label{S:grass}

\begin{lemma} \label{L:grouplike} {\em \cite[Lemma 1.7]{S}} $\Delta(\chi_\alpha)=\sum_{\beta+\gamma=\alpha}\innerprod{\beta,\gamma}^{-1}\chi_\beta\otimes\chi_\gamma.$
So $$\int_\Delta\chi_\alpha=\sum_{\beta+\gamma=\alpha}r_{\beta,\gamma}(q)x^\beta y^\gamma,$$
where $r_{\beta,\gamma}(q)$ is the rational function $\innerprod{\beta,\gamma}^{-1}r_\beta(q)\otimes r_\gamma(q)$.
\end{lemma}

Apply the Hall character $\int_\Delta$ to the identity \eqref{eq:HallID}, then we obtain the formula:
\begin{corollary} \label{C:QG}
$$\sum_{M\in\module_\alpha^{\rm ss}(Q)}\frac{|\Gr_\gamma(M)|}{a_M}=\sum_* \sum_{*'}(-1)^{s-1}\prod_{1\leqslant i<j\leqslant s}\frac{\innerprod{\gamma_i,\beta_j}}{\innerprod{\alpha_i,\alpha_j}}\prod_{k=1}^s r_{\beta_k,\gamma_k}(q),$$
where the first summation $*$ is the same as the one in \eqref{eq:HallID}, and the second summation $*'$ runs over all decomposition $\gamma_1+\cdots+\gamma_s=\gamma$ of $\gamma$ into dimension vectors such that for each $i$ $0\leqslant\gamma_i\leqslant\alpha_i=\beta_i+\gamma_i$.
\end{corollary}

We denote the above rational function by $r_{\beta,\gamma}^{\rm ss}(q)$. It follows from Lemma \ref{L:grepmoduli} and Lemma \ref{L:Euler},\ref{L:polycount} that

\begin{theorem} \label{T:QG}
Assume that $\alpha$ is coprime to $\mu$, then $(q-1)r_{\beta,\gamma}^{\rm ss}(q)$ is actually an integral polynomial, which counts the number $\sum_{M\in\Mod_\alpha^{\mu}(Q)}|\Gr_\gamma(M)|$. If $Q$ has no oriented cycles, then this polynomial has non-negative coefficients.
\end{theorem}


Using transfer matrix method as in \cite[Corollary 5.5]{R1}, we provide the following polynomial-time algorithm to compute $r_{\beta,\gamma}^{\rm ss}(q)$. The verification is almost the same as in \cite{R1}.

\begin{corollary} Let $\D{M}$ be the matrix with rows and columns indexed by the set of pairs $\{(\dot\alpha,\dot\gamma)\mid \mu(\dot\alpha)<\mu(\alpha), \dot\gamma<\min(\dot\alpha,\gamma)\} \cup \{(0,0),(\alpha,\gamma)\}\}$, where $\min$ is taken coordinatewise. The $((\alpha_i,\gamma_i),(\alpha_j,\gamma_j))$-th entry of $\D{M}$ is given by:
$\frac{\innerprod{\gamma_i,\alpha_k-\gamma_k}}{\innerprod{\alpha_i,\alpha_k}}r_{\alpha_k-\gamma_k,\gamma_k}(q)$
if $\alpha_k:=\alpha_j-\alpha_i\geqslant 0$ and $\alpha_k\geqslant\gamma_k:=\gamma_j-\gamma_i,$
and zero otherwise. Then $r_{\beta,\gamma}^{\rm ss}(q)$ is the absolute value of the cofactor of the $((0,0),(\alpha,\gamma))$-th entry.
\end{corollary}

We use this algorithm to compute an interesting example in \cite{DWZ2}.
\begin{example} Consider the four-arrow Kronecker quiver:
$$\vcenter{\xymatrix{
1 \ar@<1.5ex>[r]|{} \ar@<0.5ex>[r]|{} \ar@<-0.5ex>[r]|{} \ar@<-1.5ex>[r]|{}& 2	 }}$$

We take $\alpha=(3,4),\gamma=(1,3)$, and $\sigma=(4,-3)$, then the moduli $\Mod_\alpha^\sigma(Q)$ is a smooth geometric quotient. So we are in the situation of Lemma \ref{L:grepmoduli}: $\pi:\Mod_{(\gamma,\alpha)}^{\hat{\mu}}(A_2(Q))\twoheadrightarrow\Mod_\alpha^{\sigma}(Q)$. It is known that the general fibre of $\pi$ is a genus-$3$ curve.
$\Mod_{(\gamma,\alpha)}^{\hat{\mu}}(A_2(Q))$ has the counting polynomial:
\begin{align*}& (q + 1)(q^2 + 1)^2((q^{20}+1) + 2(q^{19}+q) + 4(q^{18}+q^2) + 7(q^{17}+q^3) + 13(q^{16}+q^4) \\
 & + 20(q^{15}+q^5) + 30(q^{14}+q^6) + 41(q^{13}+q^7) + 54(q^{12}+q^8) + 64(q^{11}+q^9) + 69q^{10}).\end{align*}
$\Mod_\alpha^{\sigma}(Q)$ has the counting polynomial:
\begin{align*}& (q^{24}+1) + (q^{23}+q) + 3(q^{22}+q^2) + 5(q^{21}+q^3) + 9(q^{20}+q^4) + 13(q^{19}+q^5) + 22(q^{18}+q^6) \\
&+29(q^{17}+q^7) + 42(q^{16}+q^8) + 52(q^{15}+q^9) + 65(q^{14}+q^{10}) + 71(q^{13}+q^{11}) + 77q^{12}.\end{align*}
\end{example}

\section{Application to Quantum Cluster Algebras} \label{S:cluster}

In this section, we assume that the quiver $Q$ has no oriented cycles. Let $E$ be the Euler matrix of $Q$, then the {\em $B$-matrix} is $B=E-E^T$. For any dimension vector $\alpha$ and $\gamma$, we set $g(\alpha)=-\alpha E^T, \phi(\gamma)=\gamma B$, and $(\cdot,\cdot)$ be the usual multiplicative antisymmetric form corresponding to $B$.

Let $\tilde{Q}$ be an extended quiver of $Q$ by a set of {\em frozen vertices} $C=(n+1,n+2,\dots,m)$, so there is no arrow from the vertices of $Q$ to $C$. The extended Euler matrix $\tilde{E}$ is by definition the left $m\times n$ submatrix of the Euler matrix of $\tilde{Q}$. We also extend the $B$-matrix to an $n\times m$ matrix $\tilde{B}$, whose $(i,j)$-th entry is
$$b_{i,j}=|\text{arrows } j\to i|-|\text{arrows } i\to j|.$$
Assume that there is an $m\times m$ antisymmetric matrix $\Lambda$ {\em unitally compatible} with $\tilde{B}$, that is, $-\tilde{B}\Lambda=(I_n,0)$. In particular, $\tilde{B}$ is of full rank. We set $\tilde{g}(\alpha)=-\alpha\tilde{E}^T$ and $\tilde{\phi}(\gamma)=\gamma\tilde{B}$, then it is easy to check that
\begin{equation} \label{eq:linear} \lambda(\tilde{g}(\alpha),\tilde{\phi}(\gamma))=\innerprod{\gamma,\alpha}, \text{  and  }  \lambda(\tilde{\phi}(\alpha),\tilde{\phi}(\gamma))=(\gamma,\alpha).\end{equation}


Let $\lambda(\cdot,\cdot)$ be the multiplicative antisymmetric form corresponding to $\Lambda$. We now change to the quantum Laurent polynomial algebra $\D{P}_\lambda$, where the multiplication rule is given by $x^\alpha x^\beta=\lambda(\alpha,\beta)^{-\frac{1}{2}}x^{\alpha+\beta}$ for any $\alpha,\beta\in\B{Z}^m$.
Following \cite{Q} (see also \cite{Ru}), we define
\begin{definition} \label{D:CV} For any indecomposable rigid object $T\in\module_\alpha(Q)$, the {\em quantum cluster variable} associated to $T$ is
$$X(T)=\sum_\gamma\innerprod{\gamma,\alpha-\gamma}^{-\frac{1}{2}}|\Gr_\gamma(T)|x^{\tilde{g}(\alpha)-\tilde{\phi}(\gamma)}.$$
The {\em quantum cluster algebra} {\em with coefficients} $\{x_v\}_{v\in\tilde{Q}\setminus Q}$ is the subalgebra of $\D{P}_\lambda$ generated by all quantum cluster variables along with $\{x_v\}_{v\in Q_0}$.
\end{definition}

One interesting problem in the cluster theory is about {\em positivity}.

\begin{corollary} \label{C:pos} For any indecomposable rigid $T$, $|\Gr_\gamma(T)|$ is counted by its Poincare polynomial, which has non-negative coefficients.
\end{corollary}

\begin{proof} If $T$ is indecomposable rigid, then the dimension vector $\alpha$ of $T$ corresponds to a {\em real Schur root}. By a result of Schofield \cite[Theorem 6.1]{So} we can always choose a slope function $\mu$ such that the general representation $T$ is $\mu$-stable. Since $\innerprod{\alpha,\alpha}_a=1$, the moduli $\Mod_\alpha^\mu(Q)$ is zero-dimensional. Then $T$ must be the only polystable representation, otherwise the moduli is reducible. Now everything follows from Theorem \ref{T:QG}.
\end{proof}

\begin{proposition}{\em [Cluster Multiplication]} \label{P:CM}
\begin{equation}\label{eq:multi} X(U)X(V)=\lambda(\tilde{g}(\alpha_1),\tilde{g}(\alpha_2))^{\frac{1}{2}}\sum_{[W]}\frac{|\Ext_Q(U,V)_W|}{|\Ext_Q(U,V)|}X(W).\end{equation}
\end{proposition}
\begin{proof} $X(U)X(V)$ \begin{align*} & =\sum_{\gamma_1,\gamma_2}\lambda(\tilde{g}(\alpha_1)-\tilde{\phi}(\gamma_1),\tilde{g}(\alpha_2)-\tilde{\phi}(\gamma_2))^{\frac{1}{2}}\innerprod{\gamma_1,\beta_1}^{-\frac{1}{2}}\innerprod{\gamma_2,\beta_2}^{-\frac{1}{2}}|\Gr_{\gamma_1}(U)||\Gr_{\gamma_2}(V)|x^{\tilde{g}(\alpha)-\tilde{\phi}(\gamma)}\\
& =\lambda(\tilde{g}(\alpha_1),\tilde{g}(\alpha_2))^{\frac{1}{2}}\sum_{\gamma_1,\gamma_2}\innerprod{\gamma,\beta}^{-\frac{1}{2}}\innerprod{\gamma_1,\beta_2}|\Gr_{\gamma_1}(U)||\Gr_{\gamma_2}(V)|x^{\tilde{g}(\alpha)-\tilde{\phi}(\gamma)} \text{\qquad \eqref{eq:linear}}\\
& =\lambda(\tilde{g}(\alpha_1),\tilde{g}(\alpha_2))^{\frac{1}{2}}\sum_{[W]}\sum_{\gamma}\innerprod{\gamma,\beta}^{-\frac{1}{2}}\frac{|\Ext_Q(U,V)_W|}{|\Ext_Q(U,V)|}|\Gr_{\gamma}(W)|x^{\tilde{g}(\alpha)-\tilde{\phi}(\gamma)} \text{\quad (Proposition \ref{P:oint})}\\
& =\lambda(\tilde{g}(\alpha_1),\tilde{g}(\alpha_2))^{\frac{1}{2}}\sum_{[W]}\frac{|\Ext_Q(U,V)_W|}{|\Ext_Q(U,V)|}X(W).
\end{align*}
\end{proof}

\begin{remark} If $U$ and $V$ are {\em adjacent}, that is, there is a representation $W$ such that both $U\oplus W$ and $V\oplus W$ are maximal rigid, then $\Ext_Q(U,V)\oplus\Ext_Q(V,U)=k$. Suppose that $\Ext_Q(U,V)=k$, then the above formula becomes the {\em quantum cluster mutation rule}.
For any (not necessary rigid) module $M$, we define $X(M)$ using the same formula as in Definition \ref{D:CV}. The multiplication formula of Lemma \ref{P:CM} still holds.
Moreover, one can extend the module category to the cluster category to incorporate the {\em initial seeds} $\{x_v\}_{v\in \tilde{Q}_0}$ into the multiplication formula.
\end{remark}

Consider the {\em principally extended quiver} $\tilde{Q}$, whose Euler matrix is of form $\sm{\ E \ \ 0 \\ -I_n \ I_n}$.
Its extended $B$-matrix $\tilde{B}=(B,I_n)$ is always unitally compatible with $\Lambda=\sm{\ \ 0\quad I_n \\ -I_n\ -B}$.
In this case, $\tilde{g}(\alpha)=(g(\alpha),\alpha),\tilde{\phi}(\gamma)=(\phi(\gamma),\gamma)$, and $\D{P}_\lambda$ is the double quantum Laurent polynomial algebra, where the multiplication is given by
\begin{align*}(x^{\alpha_1} y^{\beta_1})(x^{\alpha_2}y^{\beta_2})&=\lambda((\alpha_1,\beta_1),(\alpha_2,\beta_2))^{-\frac{1}{2}}x^{\alpha}y^{\beta}\\
&=q^{{\frac{1}{2}}(\beta_1\cdot\alpha_2-\alpha_1\cdot\beta_2)}(\beta_1,\beta_2)^{\frac{1}{2}}x^{\alpha}y^{\beta}.\end{align*}
Then the quantum cluster variables
$$X(T)=\sum_\gamma\innerprod{\gamma,\alpha-\gamma}^{-\frac{1}{2}}|\Gr_\gamma(T)|x^{g(\alpha)-\phi(\gamma)}y^{\alpha-\gamma},$$
generate the quantum cluster algebra with {\em principal coefficients}, denoted by $C_p(Q)$.
Clearly, $\int_\Delta(T)$ and $X(T)$ are related by a monomial change of variables.
Moreover, Proposition \ref{P:CM} simplifies as
\begin{equation*} \label{eq:QCMP} X(U)X(V)=\sum_{[W]}\frac{|\Ext_Q(U,V)_W|}{|\Ext_Q(U,V)|}X(W). \end{equation*}



\section{Generalization to Flags of Quiver Representations} \label{S:flag}

Let us now consider the iterated comultiplication $$\Delta^t:H(Q)\xrightarrow{\Delta} H(Q)\otimes H(Q)\xrightarrow{\Delta\otimes 1}H(Q)\otimes H(Q)\otimes H(Q)\to\cdots \to H(Q)\otimes\cdots\otimes H(Q).$$
By induction, one can easily see that
$$\Delta^t([W])=\sum_{[V_t]\dots[V_1]} F_{V_t\cdots V_1}^W\frac{a_{V_t}\cdots a_{V_1}}{a_W}[V_t]\otimes\cdots\otimes[V_1],$$
where $F_{V_t\cdots V_1}^W$ is the number of $k$-rational points of $\Fl_{V_t\cdots V_1}^W$.

We can recursively twist the multiplication in $H(Q)^{\otimes t}$ such that $\Delta^t$ is an algebra morphism. The explicit formula is visible: $$(U_t\otimes\cdots\otimes U_1)\cdot(V_t\otimes\cdots\otimes V_1)=\prod_{i>j}\innerprod{\br{U}_i,\br{V}_j} U_tV_t\otimes\cdots\otimes U_1 V_1,$$
where $\beta_i,\gamma_j$ are the dimension vectors for $U_i$ and $V_j$.
We will use the symbol $\innerprod{\beta,\gamma}_<$ (resp. $\innerprod{\beta,\gamma}_>$) to denote $\prod_{i<j}\innerprod{\beta_i,\gamma_j}$ (resp. $\prod_{i>j}\innerprod{\beta_i,\gamma_j}$).
Similarly, we can evaluate its elements in a completed multiple quantum polynomial algebra $\D{P}_{\innerprod{Q^t}}:=\D{P}_{\innerprod{Q}}^{\otimes t}$. Its multiplication is given by $(\prod x_i^{\beta_i})(\prod x_i^{\gamma_i})=\innerprod{\beta,\gamma}_>\prod x_i^{\alpha_i}$, where $\alpha_i=\beta_i+\gamma_i$. Here our convention is that $\prod x_i^{\beta_i}:=x_t^{\beta_t}\cdots x_1^{\beta_1}$. We thus get the following analogue of Proposition \ref{P:oint}

\begin{proposition} \label{P:oint_g} $\int_{\Delta^t}:=\int^{\otimes t}\circ\Delta^t$ is a $\D{P}_{\innerprod{Q^t}}$-character of $H(Q)$. We have a concrete formula: $\int_{\Delta^t}([M])=\frac{1}{a_M}\sum_\alpha|\Fl_{\alpha_t\cdots\alpha_1}(M)|\prod_{i}x_i^{\alpha_i}$ and the following identity: \begin{equation} \sum_{\beta_k+\gamma_k=\alpha_k}\innerprod{\beta,\gamma}_<|\Fl_{\beta_t\cdots\beta_1}(U)||\Fl_{\gamma_t\cdots\gamma_1}(V)|=\sum_{[W]}\frac{|\Ext_Q(U,V)_W|}{|\Ext_Q(U,V)|}|\Fl_{\alpha_t\cdots\alpha_1}(W)|.
\end{equation}
\end{proposition}

There are also analogues of Corollary \ref{C:QG}, Theorem \ref{T:QG}, and Corollary \ref{C:pos}. We summarize all in one.
Let $r_{\alpha_t\cdots\alpha_1}$ be the rational function $\innerprod{\alpha,\alpha}_>^{-1}r_{\alpha_t}\otimes\cdots\otimes r_{\alpha_1}$, then $\int_\Delta\chi_\alpha=\sum_{\alpha_t+\cdots+\alpha_1=\alpha}r_{\alpha_t\cdots\alpha_1}(q)\prod_i x_i^{\alpha_i}$.

\begin{theorem} \label{T:QF} We have \begin{enumerate}
\item[(i)] $$\sum_{M\in\module_\alpha^{\rm ss}(Q)}\frac{|\Fl_{\alpha_t\cdots\alpha_1}(M)|}{a_M}=\sum_* \sum_{*'}(-1)^{s-1}\prod_{1\leqslant i<j\leqslant s}\innerprod{\alpha_i,\alpha_j}_>\prod_{k=1}^s r_{\alpha_{tk}\cdots\alpha_{1k}}(q),$$
where the first summation $*$ is the same as the one in \eqref{eq:HallID}, and the second summation $*'$ runs over all matrices of dimension vectors $\{\alpha_{ij}\}$ such that $\sum_{j=1}^s \alpha_{ij}=\alpha_i$.

\item[(ii)] Denote the rational function in (i) by $r_{\alpha_t\cdots\alpha_1}^{\rm ss}(q)$, and assume that $\alpha$ is coprime to $\mu$, then $(q-1)r_{\alpha_t\cdots\alpha_1}^{\rm ss}(q)\in\mathbb{Z}[q]$ counts $\sum_{M\in\Mod_\alpha^{\mu}(Q)}|\Fl_{\alpha_t\cdots\alpha_1}(M)|$. If $Q$ has no oriented cycles, this polynomial has non-negative coefficients.
    In particular for a rigid module $T$, $|\Fl_{\alpha_t\cdots\alpha_1}(T)|$ has such a counting polynomial.
\end{enumerate}\end{theorem}


\section{A Functorial Generalization} \label{S:functor}

Let us consider a twisted version of Hall algebra $H(Q)$, where the multiplication is given by
$$[U][V]:=\sum_{[W]}\innerprod{\br{U},\br{V}} F_{UV}^W[W].$$ One can easily verify that this is still an associative algebra.

Motivated by \cite{BHW}, we consider the following relative setting. Let $Q'$ be another quiver, and $F$ be a functor from $\module(Q)$ to $\module(Q')$ such that
$$F \text{ induces an onto map }\Ext_Q(U,V)\xrightarrow{}\Ext_Q'(F(U),F(V))\text{ with equicardinal fibres}.$$
Standard examples of this type are certain forgetful functors. For example, let $Q'$ be a subquiver of $Q$, then $F$ is the forgetful functor induced from the algebra monomorphism: $kQ'\hookrightarrow kQ$. To simplify notations, later we will denote $F(U)$ by $U'$ and Euler form of $Q'$ by $\innerprod{\cdot,\cdot}'$.

\begin{theorem} \label{T:morphism} The map $H(F)_*: H(Q)\to H(Q')$ defined by $M\mapsto\frac{a_{M'}}{a_M}M'$ is an algebra morphism.
\end{theorem}

\begin{proof} We need to show that \begin{align*}
& \frac{a_{U'}}{a_U}U'\frac{a_{V'}}{a_V}V'=\sum_{[W]}\innerprod{\br{U},\br{V}}F_{UV}^W\frac{a_{W'}}{a_W}W'\\
\Leftrightarrow & \sum_{[M]}\innerprod{\br{U}',\br{V}'}\frac{a_{U'}}{a_U}\frac{a_{V'}}{a_V}F_{U'V'}^{M} M=\sum_{[W]}\innerprod{\br{U},\br{V}}F_{UV}^W\frac{a_{W'}}{a_W}W'\\
\Leftrightarrow & \innerprod{\br{U}',\br{V}'}\frac{a_{U'}}{a_U}\frac{a_{V'}}{a_V}F_{U'V'}^{M}=\sum_{[W]\mid W'=M}\innerprod{\br{U},\br{V}}F_{UV}^W\frac{a_{W'}}{a_W}\\
\Leftrightarrow & \frac{\Ext_{Q'}(U',V')|_M}{\Ext_{Q'}(U',V')}=\sum_{[W]\mid W'=M}\frac{\Ext_{Q}(U,V)|_W}{\Ext_{Q}(U,V)}\\
\Leftrightarrow & \frac{\Ext_{Q}(U,V)}{\Ext_{Q'}(U',V')}=\frac{\sum_{[W]\mid W'=M}\Ext_{Q}(U,V)|_W}{\Ext_{Q'}(U',V')|_M}.
\end{align*}
The last equation holds due to our assumption on the functor $F$.
\end{proof}

The map $H(F)_*$ above is constructed independently in \cite{BG} as well.
Let $Q_0$ be the quiver containing only the vertices of $Q$, and $F$ be the forgetful functor $\module(Q)\to\module(Q_0)$. Then one can easily check that $H(F)_*$ in Theorem \ref{T:morphism} is essentially the Hall character $\int$ defined in section \ref{S:def}.

Let $E$ be a rigid module in $\module(Q)$, i.e., $\Ext_Q(E,E)=0$. We denote by $E^\perp$ the left orthogonal subcategory of $\module(Q)$. Another interesting case is when $F$ is a fully faithful exact embedding of $E^\perp\hookrightarrow\module(Q)$. It is known \cite{GL} that $E^\perp$ is equivalent to $\module(Q_E)$ for another quiver $Q_E$. In this case, $\Ext_{Q_E}(U,V)\cong\Ext_Q(U',V'),a_M=a_{M'}$ and the Euler form are preserved.

\begin{corollary} The embedding $\iota:\module(Q_E)\hookrightarrow \module(Q)$ induces an embedding of algebras $H(Q_E)\hookrightarrow H(Q)$ sending $[M]$ to $[\iota(M)]$. In particular, the (coefficient-free) quantum cluster algebra $C(Q_E)$ is a subalgebra of $C(Q)$.
\end{corollary}

Dually, we can twist the comultiplication:
$$\Delta([W])=\sum_{[U][V]} \innerprod{\br{U},\br{V}}F_{UV}^W\frac{a_Ua_V}{a_W}[U]\otimes[V]=\sum_{[U][V]}\frac{|\Ext_Q(U,V)_W|}{|\Ext_Q(U,V)|}[U]\otimes[V].$$
Then readers can verify that the map $H(F)^*:H(Q')\to H(Q)$ defined by $N\mapsto\sum_{M|M'=N}M$ is a coalgebra morphism.

%
%
%

\section{Hopf structure with group-like elements} \label{S:Hopf}

In this section, we keep multiplication untwisted but comultiplication twisted as at the end of Section \ref{S:functor}. As before, the multiplication and comultiplication are not naturally compatible. But if we treat $H(Q)$ as a braided monoidal category with braiding $\psi: U\otimes V\mapsto \innerprod{\br{V},\br{U}}V\otimes U$. Then the multiplication on $H(Q)\otimes H(Q)$ is twisted as
$$(U_1\otimes V_1)\cdot(U_2\otimes V_2)=\innerprod{\br{U}_2,\br{V}_1}U_1U_2\otimes V_1V_2.$$ As before, one can verify this makes $H(Q)$ a (braided) bialgebra.

Let us consider the following map $S$ from $H(Q)$ to itself:
\begin{align*}S([W])= & -[W]+\sum_{[V_2],[V_1]}\innerprod{\gamma_2,\gamma_1}F_{V_2 V_1}^W\frac{a_{V_2}a_{V_1}}{a_W}[V_2][V_1]+\cdots+\\
& (-1)^{t}\sum_{[V_t]\dots[V_1]}\prod_{i>j}\innerprod{\gamma_i,\gamma_j}F_{V_t\cdots V_1}^W\frac{a_{V_t}\cdots a_{V_1}}{a_W}[V_t]\cdots[V_1]+\cdots.
\end{align*}
Symbolically we can think $S$ as $\sum (-1)^{i+1} m^i\Delta^i$.
It is well-known \cite[Theorem 2.14,2.18]{Hu} that by adding a degree-zero piece $K_0(Q)$, we can make a different twisted Hall algebra an honest bialgebra, then the same symbolic expression of $S$ becomes an antipode upgrading it to a Hopf algebra. In fact, the above $S$ is also an antipode for our $H(Q)$, making it a braided Hopf algebra (\cite[Definition 2.6]{KP}). The crucial identity we need to verify is that $m(S\otimes\Id)\Delta =\eta\epsilon= m(\Id\otimes S)\Delta$, but this procedure is essentially the same as in \cite{Hu}.

The antipode is useful for finding inverse of certain elements in $H(Q)$. Recall a {\em group-like element} $g$ in a coalgebra is the one satisfying $\Delta(g)=g\otimes g$ and $\epsilon(g)=1$. Let $g$ be a group-like element in a Hopf algebra, then $m(S\otimes\Id)\Delta(g)=\eta\epsilon(g)= m(\Id\otimes S)\Delta(g)$ implies that $S(g)g=1=gS(g)$. Hence $S(g)$ is the (two-sided) inverse of $g$.

Due to our twisted comultiplication, Lemma \ref{L:grouplike} becomes
\begin{lemma} $\Delta(\chi_\alpha)=\sum_{\beta+\gamma=\alpha}\chi_\beta\otimes\chi_\gamma,$
so $\Delta(\chi)=\chi \otimes \chi$. In particular, $\chi^{-1}=S(\chi)$.
\end{lemma}

An immediate consequence of $S$ being an antipode is that $S$ is a braided homomorphism.

\begin{lemma} {\em \cite[Proposition 2.10]{KP}} \label{L:antihom} $Sm = m\psi(S\otimes S) = m(S\otimes S)\psi$.
\end{lemma}

We can compose $S$ with the character $\int$ to get a new character $\Xint S:=\int S$. We have
$$\Xint S(W)=a_W^{-1}\sum_{i=0} (-1)^{i+1} F_i(W)x^\alpha,$$
where $F_i(W)$ is the number of $i$-step filtrations of $W$. It follows from Lemma \ref{L:antihom} that $\Xint S$ is a character to the usual polynomial algebra because the twist from the braiding $\psi$ balances the twist in the $\D{P}_{\innerprod{Q}}$. We denote the number $\sum_{i=0} (-1)^{i} F_i(W)$ by $F_W$.

\begin{lemma} \label{L:Fw} If $W$ is a direct sum of simples: $\bigoplus_{[S]}S^{m_S}$ and let $q_S=|\End_Q(S)|$, then $F_W=\prod_{[S]}(-1)^{m_S}q_S^{\sm{m_S\\2}}$ ; otherwise $F_W=0$.
\end{lemma}

\begin{proof} Suppose that $W=\bigoplus_{[S]}S^{m_S}$, without loss of generality we can assume our category is semisimple.
We will prove the claim by induction: let $W=U\oplus S$, then $US=\sm{m_s\\1}_{q_S}W$, so $[m_S]_{q_S}\Xint S(W)=\Xint S(U)\Xint S(S)$. Then $F_W=[m_S]_{q_S}^{-1}F_U F_S \frac{a_W}{a_U a_S}$
\begin{align*}
& =\prod_{[S]}(-1)^{m_S-1}q_S^{\sm{m_S-1\\2}}(-1)\frac{|\GL_{m_S}(\End_Q(S))|}{|\GL_{m_S-1}(\End_Q(S))||\GL_1(\End_Q(S))|}\\
& = \prod_{[S]}(-1)^{m_S}q_S^{\sm{m_S-1\\2}}\frac{[m_S]_{q_S}!(q_S-1)^{m_S}q_S^{\sm{m_s\\2}}}{[m_s]_{q_s}[m_S-1]_{q_S}!(q_S-1)^{m_S-1}q_S^{\sm{m_s-1\\2}(q_S-1)}}\\
& =\prod_{[S]}(-1)^{m_S}q_S^{\sm{m_S\\2}}.
\end{align*}

If $W$ is not semisimple, then let $\{T_i\}$ be the maximal semisimple subrepresentations of $W$. We will prove by induction on the dimension of $W$. Recall that
$$\Delta([W])=\sum_{[U][V]} \innerprod{\br{U},\br{V}}F_{UV}^W\frac{a_Ua_V}{a_W}[U]\otimes[V].$$ Since $m(\Id\otimes S)\Delta(W)=\eta\epsilon(W)=0$, we have that
$$0=m(\int\otimes\Xint S)\Delta(W)=\sum_{[U][V]} F_{UV}^W a_W^{-1} F_V x^\alpha.$$
So $F_W=\sum_{[U][V]} F_{UV}^W F_V$, where by induction $V$ runs only for semisimple proper subrepresentation of $W$.
Since any semisimple subrepresentation of $W$ is a subrepresentation of $T_i$ for some $i$. It is suffice to show that $\sum_{[U][V]} F_{UV}^{T_i} F_V=0$. Suppose that $T_i=\bigoplus_{[S]}S^{t_S}$ and $V=\bigoplus_{[S]}S^{v_S}$, then $F_{UV}^{T_i}=\prod_{[S]} \sm{t_S\\v_S}_{q_S}$. By induction,
\begin{align*} \sum_{[U][V]} F_{UV}^{T_i} F_V & =\sum_{v_S\leqslant t_S} \prod_{[S]} \sm{t_S\\v_S}_{q_S} \prod_{[S]}(-1)^{v_S}q_S^{\sm{v_S\\2}}\\
& =\prod_{[S]} \sum_{v_S\leqslant t_S}  (-1)^{v_S}q_S^{\sm{v_S\\2}} \sm{t_S\\v_S}_{q_S}\\
& = \prod_{[S]}\prod_{k=0}^{t_S-1}(1-q^k) =0. \tag{quantum binomial expansion}
\end{align*}
\end{proof}

The above lemma also holds for any finite-dimensional algebra $A=kQ/I$ because a representation of $A$ is also a representation of $Q$.
In a different direction, we can consider the full exact subcategory $\D{C}$ of $\module(Q)$. We define $F_W(\D{C})$ to be $F_W$ viewing $W$ as an object in $\D{C}$, i.e., the alternating sum of $t$-step Jordan-Holder filtration in $\D{C}$.

\begin{corollary} Let $W=\bigoplus_{[S]}S^{m_S}$ be the semisimple decomposition in $\D{C}$ and let $q_S=|\End_Q(S)|$, then $F_W(\D{C})=\prod_{[S]}(-1)^{m_S}q_S^{\sm{m_S\\2}}$ ; otherwise $F_W(\D{C})=0$.
\end{corollary}

\begin{proof} Since $\D{C}$ is full exact subcategory of $\module(Q)$, $\Hom_Q(M,N)=\Hom_{\D{C}}(M,N)$, $\Ext_Q^1(M,N)=\Ext_{\D{C}}^1(M,N)$, and in particular, $\Hom_{\D{C}}(-,-)-\Ext_{\D{C}}^1(-,-)=\innerprod{-,-}_a$. The key point is that all the crucial ingredients in the proof of Lemma \ref{L:Fw} only rely on this fact. For example, $\int$ being a character and $\chi$ being a group-like element. However, $S$ being an antipode is even independent of this fact.
\end{proof}

Now fix a slope function $\mu$ and let $\D{C}=\module_{\mu_0}(Q)$. If $W\in\module_{\mu_0}(Q)$, we will write $F_W^\mu$ instead of $F_W(\module_{\mu_0}(Q))$.

\begin{lemma} {\em \cite[Lemma 3.4]{R2}} \label{L:chiinv} $$(\chi_{\mu_0}^{\rm ss})^{-1}=\sum_{[W]\in\module_{\mu_0}^{\rm ss}(Q)} F_W^\mu[W].$$
\end{lemma}
\begin{proof} By convention, $\chi_{\mu_0}^{\rm ss}$ contains the zero representation as a summand, and hence formally invertible. We denote $x:=\chi_{\mu_0}^{\rm ss}-[0]$. Formally we have that $(\chi_{\mu_0}^{\rm ss})^{-1}=[0]-x+x^2-x^3+\cdots$, then the coefficient of $[W]$ is exactly $F_W^\mu$ by definition.
\end{proof}

\section{Generating Series}

We fix a slope function $\mu$. Consider the following series in $\D{P}_{\innerprod{Q}}:R_{\mu_0}(Q)=\int \chi_{\mu_0}^{\rm ss}$ and $R(Q)=\int \chi$. We have already seen that $\int \chi_{\mu_0}^{\rm ss}=\sum_{\mu(\alpha)=\mu_0}r_\alpha^{\rm ss}(q)x^\alpha$ and $\int \chi=\sum_\alpha r_\alpha(q) x^\alpha$. So the latter seems much easier to compute. Lemma \ref{L:HNid} implies that they are related by $R(Q)=R_{\mu_1}(Q) R_{\mu_2}(Q)\cdots R_{\mu_k}(Q)\cdots$, where $\mu_k$ increase with $k$. We can also view $R(Q)=\int \chi$ as $\int \chi_0^{\rm ss}$ with respect to the zero slope function.

A stable representation $M$ is called {\em absolutely stable} if $M\otimes_k K$ is stable for every finite field extension $k\subset K$. Let us denote $a_\alpha$ the number of $\alpha$-dimensional absolutely stable representations and $m_\alpha=|\Mod_\alpha^{\rm ss}(Q)|$.

\begin{definition} The absolute (resp. relative) local representation Poincare series of $Q$ at $\mu_0$ is $A_{\mu_0}(Q)=\sum_{\mu(\alpha)=\mu_0} a_\alpha(q)x^\alpha$ (resp. $M_{\mu_0}(Q)=\sum_{\mu(\alpha)=\mu_0} m_\alpha(q)x^\alpha$).
The global counterparts $A(Q)$ and $M(Q)$ are defined by setting $\mu=\mu_0=0$. Our convention is that the relative ones have constant term $1$, but $0$ for the absolute ones.
\end{definition}

Recall from \cite[Section 2]{MR} that $\D{P}_{\innerprod{Q}}$ has a $\lambda$-ring structure in terms of Adams operations by
$$\psi_n(f(q,x_1,\dots,x_r))=f(q^n,x_1^n,\dots,x_r^n).$$ Using Adams operations or $\sigma$-operations, we can define an exponential operator: $$\Exp(f)=\sum_{k\geq 1}\sigma_k(f)=\exp(\sum_{k\geq 1}\frac{\psi_k(f)}{k}).$$ It has an inverse $\Log$ given by
$$\Log(f)=\sum_{k\geq 1}\frac{\mu(k)}{k}\psi_k(\log(f)),$$ where $\mu$ is the classical m\"{o}bius function. All multiplications involved in defining the above operations are the usual multiplication for polynomials.
The series $M_{\mu_0}(Q)$ has a form of zeta-function:
\begin{lemma} {\em \cite[Theorem 8.3]{R3}}
$M_{\mu_0}(Q)=\Exp(A_{\mu_0}(Q))$.
\end{lemma}


Applying the character $\int$ to the identity in Lemma \ref{L:chiinv}, it is proven in \cite{MR} that $A_{\mu_0}(Q)$ and $R_{\mu_0}(Q)$ are related in $\D{P}_{\innerprod{Q}}$.

\begin{theorem} {\em \cite[Theorem 4.1]{MR}} \label{T:GS} $R_{\mu_0}(Q)\Exp(\frac{A_{\mu_0}(Q)}{1-q})=1$. In particular, both $a_\alpha$ and $m_\alpha$ are polynomials in $q$.
\end{theorem}

\begin{example} If $Q$ has no oriented cycles, then the only simple representations are the one-dimensional ones corresponding to vertices, so $A(Q)=\sum_{v\in Q_0} x_v$, and thus $M(Q)=\prod_{v\in Q_0}(1-x_v)^{-1}$.
On the other hand, if $Q$ has only one vertex but $m$ loops, some interesting results are obtained in \cite{MR}.

For the ``next" case: $m$-arrow Kronecker quivers $K_m$, we know from a result of T. Weist \cite[Corollary 6.14]{W} that the Euler characteristic of the moduli space of $(d,d)$-dimensional stable representations of $K_m$ is $m$ for $d=1$ and zero otherwise. Since semi-stable representations can be counted by symmetric products of stable representations, we see that the Euler characteristic of the moduli space of $(d,d)$-dimensional stable representations of $K_m$ is the binomial coefficients:
$$\chi_c(\Mod_{(d,d)}^{\rm ss}(K_m),\B{Q}_l) =\sm{d+m-1\\m-1}.$$
Formally we can consider the generating series $\sum_{d} m_\alpha(K_m,q)x^{(d,d)}y^m$. The above implies that it equals $(1-(x^{(1,1)}+y))^{-1}$ at $q=1$. In other words, $\chi_c(\Mod_{(d,d)}^{\rm ss}(K_m),\B{Q}_l)$ is the $d$-th element of the $d+m-1$-th floor of Jia Xian's triangle\footnote{Jia Xian is 600 years younger than Pascal.}. However, before specializing at $q=1$, it is clearly not the usual quantum Jia Xian's triangle. The symmetry is already broken at level five:
$$1,[5],(q^2+1)(q^7+q^6+q^5+q+1),(q+1)(q^9+q^7+q^4+q^2+1),[5],1.$$
\end{example}


Finally we come back to the usual quantum setting: let $\D{P}_{(Q)}$ be the usual quantum polynomial algebra. Then an analog of Example \ref{Ex:int} says that $\int:[M]\mapsto \innerprod{\alpha,\alpha}^{\frac{1}{2}}\frac{t^\alpha}{a_M}$ is an $\D{P}_{(Q)}$-character of $H(Q)$.
When the subcategory $\module_{\mu_0}$ contains only one simple objects $S$, the generating series $R_{\mu_0}(Q)$ is of the form of quantum exponential. Recall the quantum exponential $\exp_q(z):=\sum_{n=0}^{\infty}\frac{z^n}{[n]!}$, and we denote $\B{E}(z):=\exp_q(\frac{q^{1/2}}{q-1}z)$. Readers can easily verify that $R_{\mu_0}(Q)=\B{E}(x^{\br{S}})$.

For Dynkin quivers, one can find at least two slope functions $\mu_1,\mu_2$ such that for any fixed number $\mu_0$, the subcategory $\module_{\mu_0}(Q)$ contains at most one simple object.
For $\mu_1$, the only stable objects are the simples $\{S_v\}_{v\in Q_0}$ have strictly decreasing slopes;
for $\mu_2$, the stable objects $\{E_i\}_{i=1\dots N}$ are precisely the indecomposable representations and their slopes increase from left to right in the Auslander-Reiten quiver. The existence of $\mu_2$ is nontrivial. It appears in an unpublished work of Hille-Juteau.

\begin{proposition} \label{P:dilog} {\em \cite{R4},\cite[Theorem 1.6]{Ke}} Let $Q$ be a Dynkin quiver, then we have
$$\B{E}(x^{\delta_1})\cdot\cdots\cdot\B{E}(x^{\delta_n})=\B{E}(x^{\epsilon_1})\cdot\cdots\cdot\B{E}(x^{\epsilon_N}),$$
where $\{\delta_i\}$ is the standard basis of $\mathbb{Z}_n$ and $\{\epsilon_j\}$ is the dimension vector for $E_j$.
\end{proposition}

\section*{Acknowledgement}
The author always feels very grateful to his adviser Harm Derksen. The author also wants to thank Professor Sergey Fomin for encouraging him at a critical time. Finally a special thanks to Professor Markus Reineke, who paved all the way and critically commented on the manuscript. Without them, nothing is possible.

%

\bibliographystyle{amsplain}

\end{document}